\documentclass[11 pt,a4paper,twoside,reqno]{amsart}
\usepackage{amsfonts,amssymb,amscd,amsmath,enumerate,verbatim,calc}
\usepackage{mathtools}
\usepackage{enumerate}
\usepackage{setspace} 

\usepackage[margin=0.8in]{geometry}
\usepackage{hyperref}
\newtheorem{theorem}{Theorem}[section]
\newtheorem{lemma}[theorem]{Lemma}
\newtheorem{proposition}[theorem]{Proposition}
\theoremstyle{definition}

\newtheorem{corollary}[theorem]{Corollary}
\newtheorem{conjecture}[theorem]{Conjecture}
\newtheorem{ex}[theorem]{Example}
\newtheorem{remark}[theorem]{Remark}
\newtheorem{question}[theorem]{Question}

\def\N{{\mathbb N}}

\def\R{{\mathbb R}}
\def\RP{{\mathbb R}\mathrm{P}}
\def\Z{{\mathbb Z}}
\def\cB{{\mathcal B}}

\def\P{{\mathcal P}}

\newcommand{\vr}[2]{\mathrm{VR}(#1;#2)}

\newcommand{\cech}[2]{\mathrm{\check{C}}(#1;#2)}
\newcommand{\acech}[3]{\mathrm{\check{C}}(#1,#2;#3)}

\DeclareMathOperator{\conn}{conn}
\DeclareMathOperator{\st}{st}
\DeclareMathOperator{\diam}{diam}

\DeclareMathOperator{\cov}{cov}
\DeclareMathOperator{\CN}{CN}
\DeclareMathOperator{\numcov}{numCover}
\DeclareMathOperator{\numpack}{numPack}
\newcommand{\bor}[2]{\mathrm{Bor}(#1;#2)}
\newcommand{\cbor}[2]{\mathrm{Bor}[#1;#2]}

\DeclareMathOperator{\hdim}{h-dim}

\usepackage{xcolor}

\usepackage{color}

\title{Homotopy connectivity of \v{C}ech complexes of spheres}

\author{Henry Adams}
\address{Henry Adams, Department of Mathematics, University of Florida, 358 Little Hall, Gainesville, FL 32611, USA.}
\email{henry.adams@ufl.edu}

\author{Ekansh Jauhari}
\address{Ekansh Jauhari, Department of Mathematics, University of Florida, 358 Little Hall, Gainesville, FL 32611, USA.}
\email{ekanshjauhari@ufl.edu}

\author{Sucharita Mallick}
\address{Sucharita Mallick, Department of Mathematics, University of Florida, 358 Little Hall, Gainesville, FL 32611, USA.}
\email{sucharitamallick@ufl.edu}

\date{\today}

\subjclass[2020]
{Primary 
55U10, 
Secondary 
55Q52, 
05E45, 
52C17. 
}

\keywords{\v{C}ech complex, homotopy connectivity, Borsuk graph, chromatic number, homological dimension, neighborhood complex.}

\begin{document}

\begin{abstract}
%
Let $S^n$ be the $n$-sphere with the geodesic metric and of diameter $\pi$.
The intrinsic \v{C}ech complex $\cech{S^n}{r}$ is the nerve of all open balls of radius $r$ in $S^n$.
In this paper, we show how to control the homotopy connectivity of $\cech{S^n}{r}$ in terms of coverings of spheres.
Our upper bound on the connectivity, which is sharp in the case $n=1$, comes from the chromatic numbers of Borsuk graphs of spheres.
Our lower bound is obtained using the conicity (in the sense of Barmak) of $\cech{X}{r}$ for sufficiently dense, finite subsets $X$ of $S^n$.
Our bounds imply the new result that for $n\ge 1$, the homotopy type of $\cech{S^n}{r}$ changes infinitely many times as $r$ varies over $(0,\pi)$; we conjecture only countably many times.
Additionally, we lower bound the homological dimension of \v{C}ech complexes of finite subsets $X$ of $S^n$ in terms of packings of $X$.
\end{abstract}

\maketitle



\section{Introduction}
\v{C}ech complexes are among the most important tools in applied topology for approximating the shape of data.
Indeed, whereas a finite dataset has no interesting topology on its own, if one replaces each point in the dataset with a ball of radius $r$, then the union of the balls serves as a proxy for the shape of the data.
Even better, by letting the radius $r$ vary from small to large, one can use persistent homology to get a multi-resolution summary of the shape of the data~\cite{Carlsson2009}.
Though the union of balls is difficult to store on a computer, the \v{C}ech complex is a combinatorial representation with a vertex for each data point, an edge whenever two balls intersect, a triangle whenever three balls intersect, etc.
Since balls in Euclidean space are convex, the nerve lemma implies that this \v{C}ech complex is homotopy equivalent to the union of the balls (see~\cite{Borsuk1948} or~\cite[Corollary~4G.3]{Hatcher}).

The theory of \v{C}ech complexes for balls in Euclidean space is well-understood.
But what about for balls in other manifolds?
In a Riemannian manifold $M$, a ball of radius $r$ need only be geodesically convex up to a certain scale $r$ (called the \emph{convexity radius}) depending on the curvature of the manifold.
For larger scales, the intersection of two balls may be disconnected, meaning that the hypotheses of the nerve lemma may no longer be satisfied.
What are the homotopy types of intrinsic \v{C}ech complexes of manifolds, where \emph{intrinsic} means that the \v{C}ech complex is the nerve of balls restricted to live in the manifold?
In this paper, we initiate the study of this question for the simplest non-contractible manifolds: spheres.
We hope our paper inspires follow-up work on the intrinsic \v{C}ech complexes of other manifolds.
Furthermore, instead of the statistical context of a finite dataset $X$ sampled from the $n$-sphere $S^n$, 
we begin with the mathematical context of the \v{C}ech complex of the entire $n$-sphere.
If one's dataset $X$ is sampled from a manifold $M$, then in the limit, as more data points are sampled, the stability of persistent homology~\cite{ChazalDeSilvaOudot2014} implies that the intrinsic \v{C}ech persistent homology of the dataset $X$ converges to the intrinsic \v{C}ech persistent homology of the entire manifold $M$.

What are the homotopy types of intrinsic \v{C}ech complexes of spheres?
Equip the $n$-sphere $S^n$ with the geodesic metric of diameter $\pi$.
If the scale $r$ satisfies $0<r\le\frac{\pi}{2}$, then the open balls $\{B_{S^n}(x;r)\}_{x\in S^n}$ form a good cover of the sphere $S^n$, meaning that any intersection of balls is either empty or contractible.
Since $S^n$ is paracompact as it is a metric space, the nerve lemma implies that the intrinsic \v{C}ech complex $\cech{S^n}{r}\coloneqq \mathrm{Nerve}(\{B_{S^n}(x;r)\}_{x\in S^n})$ is homotopy equivalent to $S^n$ (\cite[Theorem~4G.3]{Hatcher}; see also~\cite[Theorem~13.1.3]{tom2008algebraic}).
But what are the homotopy types 
of $\cech{S^n}{r}$ when $\frac{\pi}{2}<r<\pi$?

The case $n=1$ is well-understood by~\cite[Section~9]{AA-VRS1}: as the scale $r$ increases, the complexes $\cech{S^1}{r}$ obtain the homotopy types of $S^1$, $S^3$, $S^5$, $S^7$, \ldots.
We also refer the reader to~\cite{AAFPP-J} for a description of the homotopy type of $\cech{X}{r}$ for $X$ an arbitrary finite subset of $S^1$, and to~\cite{lim2026strange} for a description of the expected Euler characteristic of $\cech{X}{r}$ for $X$ a finite subset chosen uniformly at random from $S^1$.
However, for $n\ge 2$, essentially nothing is known about the homotopy types of $\cech{S^n}{r}$ for $\frac{\pi}{2}<r<\pi$, except that they are simply connected by~\cite[Theorem 11.2]{virk20201}.
In particular, though the first new homotopy type beyond $S^n$ (i.e., the homotopy type of $\cech{S^n}{\frac{\pi}{2}+\varepsilon}$ for $\varepsilon>0$ sufficiently small) seems within reach, to our knowledge no conjecture for this homotopy type has appeared in the literature.
Can we bound the homotopy connectivity of $\cech{S^n}{r}$ as a function of the scale $r$?

We denote the homotopy connectivity of a topological space $X$ by $\conn(X)$, so $\pi_i(X)$ vanishes for all $i\le \conn(X)$.
For a metric space $(X,d)$, the $k$-th covering radius, denoted $\cov_X(k)$, is the infimum $r\ge 0$ such that there exist $k$ closed balls of radius $r$ that cover $X$.
We show in Proposition~\ref{prop:fintoinf}, using Barmak's Appendix~\cite{farber2023large}, that if $0<\delta<\cov_{S^n}(2k+2)$, then $\conn(\cech{S^n}{\pi-\delta}) \ge k$.
Next, using Lov{\'a}sz's classical bound on the chromatic number of a graph in terms of the connectivity of its neighborhood complex~\cite{Lovasz1978}, we upper bound the connectivity of \v{C}ech complexes of spheres.
The (open) Borsuk graph on $S^n$ at scale $\delta>0$, denoted $\bor{S^n}{\delta}$, is the infinite graph with vertex set $S^n$ whose edges are the pairs of vertices $v,w$ which are more than $\delta$ apart in the geodesic distance, i.e., $d(v,w)>\delta$.
The neighborhood complex of the Borsuk graph $\bor{S^n}{\delta}$ is the \v{C}ech complex $\cech{S^n}{\pi-\delta}$.
As a result, in Proposition~\ref{prop:upper}, we show that if $\delta> 2\cov_{S^n}(k+1)$, then we have that $\conn(\cech{S^n}{\pi-\delta})\leq k-2$.

When combined together, we get strong bounds on the connectivity of \v{C}ech complexes of spheres.

\begin{theorem}
\label{thm:main}
For $n\ge 1$ and $\delta\in(0,\pi)$, if $\conn(\cech{S^n}{\pi-\delta})=k-1$, then
\[\cov_{S^n}(2k+2)\le \delta\le 2\cov_{S^n}(k+1).\]
\end{theorem}

See Figure~\ref{fig:barsS2} for an illustration of this theorem in the case $n=2$.

\begin{figure}[htb]
\begin{center}
\includegraphics[width=5in]{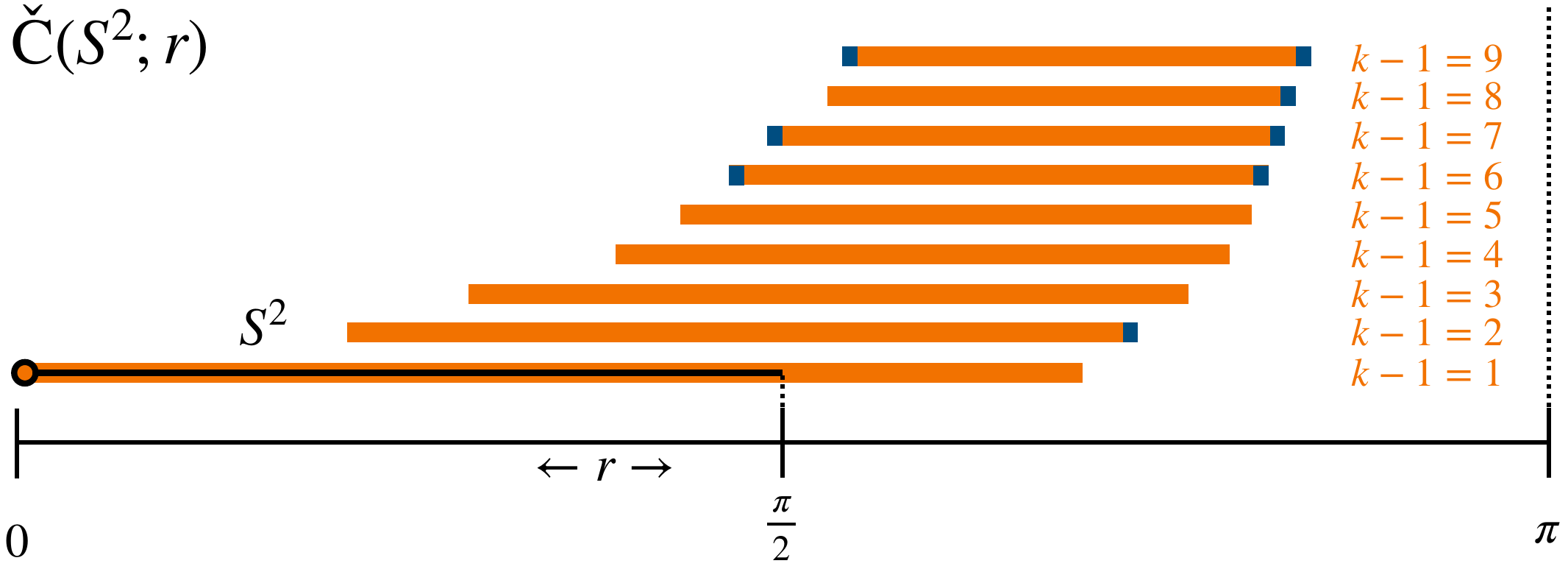}
\end{center}
\caption{Intervals where $\cech{S^2}{r}$ may have connectivity $k-1$, given by Theorem~\ref{thm:main}.
The blue endpoints are plotted using only approximate values of $\cov_{S^2}(2k+2)$ or of $2\cov_{S^2}(k+1)$; see~\cite{tarnai1991covering}.}
\label{fig:barsS2}
\end{figure}

For a metric space $(X,d)$, the $r$-covering number, denoted $\numcov_X(r)$, is the smallest integer $n$ such that $X$ can be covered by $n$ closed balls of radius $r>0$.
Rearranging the inequalities in Theorem~\ref{thm:main} gives the following corollary.

\begin{corollary}
\label{cor:main}
For $n\ge 1$ and $\delta\in(0,\pi)$, the connectivity of $\cech{S^n}{\pi-\delta}$ satisfies
\[\tfrac{1}{2}\numcov_{S^n}(\delta) - 2 \le \conn(\cech{S^n}{\pi-\delta}) \le \numcov_{S^n}\left(\tfrac{\delta}{2}\right)-2.\]
\end{corollary}

From~\cite[Corollary 1.2]{ABV}, we know that the connectivity of \emph{Vietoris--Rips} complexes of sphere satisfies
\[\tfrac{1}{2}\numcov_{S^n}(\delta) - 2 \le \conn(\vr{S^n}{\pi-\delta}) \le \numcov_{\RP^n}\left(\tfrac{\delta}{2}\right)-2\]
for each $n\ge 1$ and $\delta>0$.
So, while the current best bounds on $\conn(\vr{S^n}{r})$ depend on the covering numbers of both $\RP^n$ and $S^n$, the current best bounds on $\conn(\cech{S^n}{r})$ depend only on the covering numbers of $S^n$.

It follows as a consequence that the homotopy type of $\cech{S^n}{r}$ changes infinitely many times.

\begin{corollary}
For $n\ge 1$, the homotopy type of the \v{C}ech complex $\cech{S^n}{r}$ changes infinitely many times as $r$ varies over the range $(0,\pi)$.
\end{corollary}

\begin{proof}
Let $r=\pi-\delta$ for $\delta\in(0,\pi)$.
Note that $\numcov_{S^n}(\delta)\to\infty$ as $\delta\to 0$.
So, from the left-hand inequality in Corollary~\ref{cor:main}, we see that as $\delta$ tends to $0$, the connectivity of $\cech{S^n}{\pi-\delta}$ tends to infinity.
On the other hand, since $\numcov_{S^n}(\delta)<\infty$ for all $\delta>0$, we get from the right-hand inequality in Corollary~\ref{cor:main} that the connectivity of $\cech{S^n}{\pi-\delta}$ is always finite.
Hence, as $r$ varies between $0$ and $\pi$, $\conn(\cech{S^n}{\pi-\delta})$ attains infinitely many integer values.
\end{proof}

We conjecture that the homotopy type of $\cech{S^n}{r}$ changes only countably many times as $r$ varies; see Conjecture~\ref{conj:countably}.

How good are our bounds on $\cech{S^n}{\pi-\delta}$?
We are able to test this in the case of $n=1$.
From~\cite[Theorem 9.8]{AA-VRS1}, we know that $\cech{S^1}{r}$ obtains the homotopy types $S^1, S^3, S^5, S^7, \ldots$ as $r$ increases.
Thus, we compare the scales obtained in Theorem~\ref{thm:main} with the actual scales of those connectivities, and it turns out that for $n=1$, the lower bound on $\delta$ in the theorem is off by approximately a multiplicative factor of $4$, whereas the upper bound on $\delta$ is tight.
See Figure~\ref{fig:barsS1}.

\begin{figure}[htb]
\begin{center}
\includegraphics[width=\textwidth]{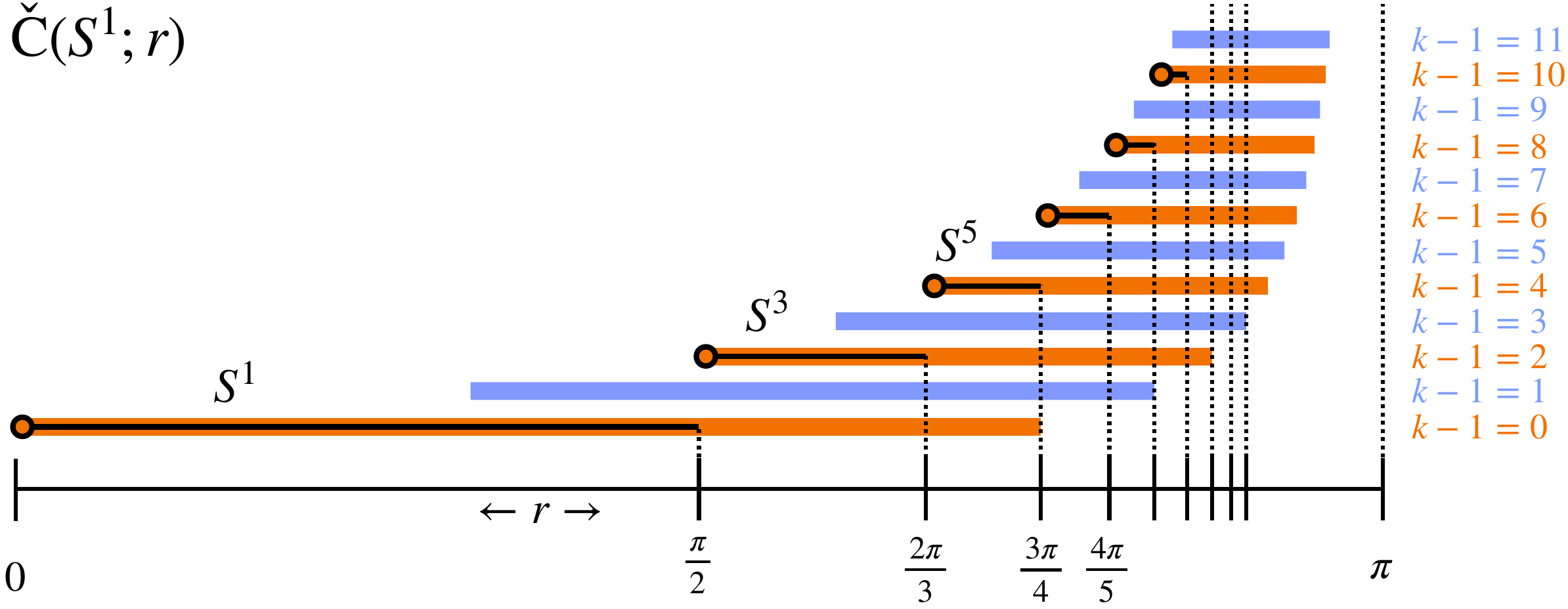}
\end{center}
\caption{The homotopy types of $\cech{S^1}{r}$ as $r$ varies~\cite{AA-VRS1} are indicated by black bars.
Theorem~\ref{thm:main} gives intervals where $\cech{S^1}{r}$ may have connectivity $k-1$, which are indicated by colored bars.
Left endpoints of orange bars (when $k-1$ is even) are tight.}
\label{fig:barsS1}
\end{figure}

We also study the homological dimension of \v{C}ech complexes of (finite subsets) of spheres.
The $\Z$-homological dimension of a topological space $Y$, denoted $\hdim_{\Z}(Y)$, is the largest integer $i$ such that $\widetilde{H}_i(Y; \Z) \neq 0$.
We give a lower bound on the $\Z$-homological dimension of $\cech{X}{\pi -\delta}$ for all finite subsets $X\subseteq S^n$ in terms of the packings of $X$, and we ask if these lower bounds extend to $\cech{S^n}{\pi -\delta}$.

The paper is organized as follows.
In Section~\ref{sec:related}, we survey some related work on \v{C}ech complexes of spheres, and in Section~\ref{sec:prelim}, we introduce the definitions and notations that we will use in this paper.
We give a lower bound on the homotopy connectivity of \v{C}ech complexes of spheres using Barmak's work relating homotopy connectivity to simplicial conicity in Section~\ref{sec:lower}.
We upper bound the homotopy connectivity of \v{C}ech complexes of spheres using Lov{\'a}sz's seminal result relating the chromatic number of a graph to the connectivity of its neighborhood complex in Section~\ref{sec:chromatic}.
In Section~\ref{sec:circle}, we describe what our bounds on connectivity give in the case of the circle.
In Section~\ref{sec:hdim}, we bound the homological dimension of \v{C}ech complexes of finite subsets of spheres from below using the theory of box complexes, and we study the problem of extending our lower bounds to \v{C}ech complexes of the entire spheres.
We conclude this paper in Section~\ref{sec:conclusion} with a list of open questions and conjectures.

\section{Related work}
\label{sec:related}

In~\cite{Lovasz1978}, Lov{\'a}sz famously used topological arguments, and in particular, the Borsuk--Ulam theorem, to lower bound the chromatic numbers of Kneser graphs.
Though Kneser graphs are finite graphs of a combinatorial flavor, Lov{\'a}sz's techniques for understanding their chromatic numbers were partially inspired by the infinite and geometrically flavored Borsuk graphs (see~\cite[Footnote~1]{Lovasz1978}), which have $n$-spheres as their vertex sets.
Lov{\'a}sz proved that the chromatic number of \emph{any} graph $G$ satisfies $\chi(G)\ge\conn(N(G))+3$, where $N(G)$ is the neighborhood simplicial complex of $G$.
He further showed that this general lower bound turns out to be an equality in the case of Kneser graphs.

The neighborhood complexes of Borsuk graphs are easily identified to be \v{C}ech complexes of spheres.
Indeed, the Borsuk graph $\bor{S^n}{\delta}$ has $S^n$ as its vertex set and an edge between any two vertices at a distance more than $\delta$.
Its neighborhood complex turns out to be the nerve simplicial complex of all open balls in $S^n$ of radius $\pi-\delta$, recovering the \v{C}ech complex $\cech{S^n}{\pi-\delta}$ on the nose.
Hence, by choosing $G=\bor{S^n}{\delta}$ in Lov{\'a}sz's bound, we relate the connectivity of \v{C}ech complexes of spheres to the chromatic numbers of Borsuk graphs:
\[
\chi(\bor{S^n}{\delta}) \ge \conn(N(\bor{S^n}{\delta}))+3 = \conn(\cech{S^n}{\pi-\delta})+3.
\]
This recovers the classical result that the chromatic number of the Borsuk graph is at least $n+2$ for scales $\delta$ which are only slightly less than $\pi$~\cite{lovasz1983self,raigorodskii2010chromatic,raigorodskii2012chromatic}\footnote{In these references, the result for chromatic numbers is proven for closed Borsuk graphs $\cbor{S^n}{\delta}$, which have an edge between vertices $v$ and $w$ whenever $d(v,w)\ge \delta$.
However, this result remains true (see Remark~\ref{rem:opencloseborsuk}) when Borsuk graphs are considered with the ``open" convention~\cite{erdos-hajnal1967} used in this paper.}, since in this regime, the nerve lemma~\cite{Borsuk1948,Hatcher} applies to give $\cech{S^n}{\pi-\delta}\simeq S^n$, yielding $\chi(\bor{S^n}{\delta}) \ge \conn(S^n)+3 = n+2$.
There is furthermore an equality $\chi(\bor{S^n}{\delta}) = n+2$ for $\delta$ sufficiently close to $\pi$.
However, in this paper, we allow the scale $\delta$ to be far away from $\pi$.
Using known bounds on the chromatic numbers of {Borsuk} graphs in terms of packings and coverings of spheres (see, for example,~\cite{MoyPhD}), we obtain upper bounds on the connectivity of \v{C}ech complexes of spheres at all scale parameters.

To lower bound the connectivity of \v{C}ech complexes of spheres, we rely on the appendix by Barmak in Farber's paper~\cite{farber2023large}.
This appendix builds upon prior work~\cite{meshulam2001clique,meshulam2003domination,chudnovsky2000systems,Kahle2009} lower bounding the connectivity of clique complexes of graphs.
Crucially, Barmak's appendix generalizes these results to arbitrary simplicial complexes, meaning they can be applied to the non-clique \v{C}ech simplicial complexes we study here.
We refer the reader to related results on the connectivity of Vietoris--Rips complexes of hypercube graphs~\cite{bendersky2023connectivity}, on the connectivity of clique complexes of preferential attachment graphs~\cite{siu2024topological}, on the connectivity of Vietoris--Rips complexes of spheres~\cite{ABV}, and on the connectivity of independence complexes of graphs~\cite{adamaszek2011lower} and hypergraphs~\cite{taylan2024bounds}.
See~\cite{bobrowski2017vanishing,bobrowski2022homological,bobrowski2019random,lim2026strange} for results on the topology of random \v{C}ech complexes,~\cite{virk20201} for some applications of the fundamental group functor to the \v{C}ech filtration of compact geodesic spaces, and~\cite{kahle2020chromatic} for results on the chromatic numbers of random Borsuk graphs. 

Researchers in applied topology use \v{C}ech and Vietoris--Rips complexes to approximate the shape of a dataset.
If one's dataset is sampled from a manifold, then in the limit, as more and more points are sampled, the persistent homology of the sample converges to the persistent homology of the manifold.
For small scales, Hausmann's theorem~\cite{Hausmann1995} implies that the Vietoris--Rips complex of the manifold is homotopy equivalent to the manifold, and the nerve lemma~\cite{Borsuk1948} implies that the intrinsic \v{C}ech complex of the manifold is homotopy equivalent to the manifold.
However, we do not yet have a good mathematical understanding of the behavior of the Vietoris--Rips or \v{C}ech complexes of manifolds at larger scales, even though the whole idea of persistent homology is to allow the scale to vary from small to large.
{Determining the homotopy types of Vietoris--Rips complexes of spheres is a difficult but largely open problem}~\cite{AA-VRS1,lim2022vietoris,ABV}, which has important connections to Gromov--Hausdorff distances between unit spheres of different dimensions~\cite{lim2022gromov,GH-BU-VR,harrison2023quantitative}.
Though the homotopy types of intrinsic \v{C}ech complexes of spheres have been studied less, we think this may be an even more natural problem.
Indeed, these homotopy types can be viewed as generalizations of the nerve lemma, for one may ask:
What are the homotopy types of the nerve of balls in a sphere (or more generally in a manifold) when the radii are now large enough so that the hypotheses of the nerve lemma are no longer satisfied?

We remark that the homotopy types of the intrinsic \v{C}ech complexes of the circle are fully understood by~\cite[Theorem 9.8]{AA-VRS1}.
Indeed, we have $\cech{S^1}{r} \simeq S^{2k+1}$ for $\frac{\pi k}{k+1} < r \le \frac{\pi (k+1)}{k+2}$ for $k\in \N$.
The way to interpret these homotopy types is as follows.
For $0<r\le \frac{\pi}{2}$, open balls of radius $r$ (circular arcs of length $2r$) in the circle are geodesically convex, and hence, the nerve lemma applies to give $\cech{S^1}{r}\simeq S^1$.
However, once $r$ exceeds $\frac{\pi}{2}$, a ball of radius $r$ is no longer geodesically convex, and indeed, the intersection of two antipodal balls is disconnected.
Therefore, the hypotheses of the nerve lemma are not satisfied.
So, up to homotopy type, a circle's worth of disks are attached, giving $\cech{S^1}{r}\simeq S^1*S^1=S^3$ for $\frac{\pi}{2} < r \le \frac{2\pi}{3}$.
Once $r$ exceeds $\frac{2\pi}{3}$, the intersection of three evenly spaced balls now has three connected components, and the homotopy type becomes $(S^1*S^1)*S^1=S^3*S^1=S^5$.
More generally, for $\frac{\pi k}{k+1} < r \le \frac{\pi (k+1)}{k+2}$, we have that $\cech{S^1}{r}$ is homotopy equivalent to a $(k+1)$-fold join of circles $S^1*\cdots*S^1=S^{2k+1}$.

\section{Preliminaries}
\label{sec:prelim}

In this section, we provide preliminary definitions and notations that we will use for metric spaces, topological spaces, simplicial complexes, and graphs.
Important concepts will include covering radii and numbers, connectivity, \v{C}ech complexes, Borsuk graphs, and neighborhood complexes.

\subsection{Metric spaces}
\label{ssec:metric}

Let $(X,d)$ be a metric space.
Further, let $B(x,r)$ be the open ball of radius $r$ centered at $x$, i.e., $B(x;r)=\{y\in X \mid d(y,x)<r\}$.
Also, let $B[x,r]=\{y\in X \mid d(y,x)\leq r\}$ be the closed ball of radius $r$ centered at $x$.
{On occasion, we may write $B_X(x,r)$ or $B_X[x,r]$, for clarity.}

Our bounds on the connectivity of \v{C}ech complexes of spheres will be in terms of the $k$-th covering radius or the $\delta$-covering number of a metric space $(X,d)$, which we now define.

For $k\ge 1$, the \emph{$k$-th covering radius} of a metric space $(X,d)$, denoted $\cov_X(k)$, is defined as 
\[
\cov_X(k) = \inf\left\{r\ge 0 \hspace{1mm} \middle| \hspace{1mm} \text{ there exist } x_1,\ldots,x_k \in X \text{ such that } \bigcup_{i=1}^k B[x_i,r] = X \right\}.
\]
When $X$ is compact, this infimum is attained (see~\cite[Lemma~3.1]{ABV}), in which case $\cov_X(k)$ is the smallest radius such that $k$ closed balls of that radius cover $X$.

We can also ask how many balls of a fixed radius are required to cover a metric space.
For $\delta>0$, the \emph{$\delta$-covering number} of a metric space $(X,d)$, denoted $\numcov_X(\delta)$, is defined as 
\[
\numcov_X(\delta) = \min\left\{n \ge 1 \hspace{1mm} \middle| \hspace{1mm} \text{ there exist } x_1,\ldots,x_n \in X \text{ such that } \bigcup_{i=1}^n B[x_i,\delta] = X \right\}.
\]
Analogously, we can ask how many points in $X$ can be chosen such that any two points are more than a specified distance apart.
For $\delta>0$, the $\delta$-\emph{packing number} of a metric space $(X,d)$, denoted $\numpack_{X}(\delta)$, is defined as
\[
\numpack_{X}(\delta)=\max\left\{m \ge 1 \hspace{1mm} \middle| \hspace{1mm} \text{ there exist } x_1,\ldots,x_m \in X \text{ such that } d(x_i,x_j) > \delta \text{ for all } i,j \right\}.
\]
We note that if the diameter of $X$ is infinite, then both $\numcov_X(\delta)$ and $\numpack_X(\delta)$ are infinite for any $\delta> 0$.

The main metric space of interest in this paper is the $n$-sphere $S^n=\{x\in\R^{n+1}~|~\|x\|=1\}$, equipped with the geodesic metric.
In other words, the distance between two points $x,y\in S^n$ is given by $d(x,y)=\arccos(x\cdot y)\in[0,\pi]$.
With this metric, the sphere $S^n$ has diameter $\pi$.

\subsection{Topological spaces}

We write $Y\simeq Y'$ to denote that the two topological spaces $Y$ and $Y'$ are homotopy equivalent.
The \emph{connectivity} of a non-empty topological space $Y$, denoted $\conn(Y)$, is the maximal index $k$ such that the homotopy groups $\pi_i(Y)$ are trivial for all $i\le k$.

For a commutative ring $R$ with unity, the \emph{$R$-homological dimension} of a non-empty topological space $Y$, denoted $\hdim_R(Y)$, is the maximal degree $i$ such that the $i$-th reduced homology group of $Y$ with coefficients in $R$ is non-zero, i.e., $\widetilde{H}_i(Y; R)\neq 0$.

We note that connectivity and $R$-homological dimensions are homotopy invariants of topological spaces, i.e., if $Y\simeq Y'$, then $\conn(Y)=\conn(Y')$ and $\hdim_R(Y)=\hdim_R(Y')$.

\subsection{Simplicial complexes}

A \emph{simplicial complex $K$} on a vertex set $V$ is a collection of {finite} subsets of $V$, containing all singletons, such that if $\sigma\in K$ and $\tau\subseteq \sigma$, then $\tau \in K$ as well.
In this paper, we identify simplicial complexes with their geometric realizations. 
We will be most interested in simplicial complexes that can be built on top of metric spaces, such as \v{C}ech and Vietoris--Rips simplicial complexes.

A \emph{\v{C}ech complex} on a metric space $(X,d)$ for scale $r > 0$, denoted $\cech{X}{r}$, is the nerve of the collection $\{B(x,r)\}_{x \in X}$ of open balls of radius $r$ in $X$, indexed by the points of $X$.
More explicitly,
\[
\cech{X}{r} = \left\{ \text{finite }\sigma \subseteq X\ \middle|\ \bigcap_{x \in \sigma}B(x,r) \ne \emptyset \right\}.
\]
For $k\ge 0$ and $x_0,\ldots,x_k\in X$, we have that $\sigma = [x_0,\ldots,x_k]$ is a simplex in $\cech{X}{r}$ if there exists some $z_{\sigma} \in X$ such that $x_i \in B(z_{\sigma},r)$ for all $0 \le i \le k$, i.e., if the vertex set of $\sigma$ is contained in an open ball of radius $r$.

A \emph{Vietoris--Rips complex} on a metric space $(X,d)$ for scale $r > 0$, denoted $\vr{X}{r}$, is a simplicial complex having $X$ as its vertex set and a finite set $\sigma \subseteq X$ as a simplex if $\diam(\sigma) < r$; see~\cite{Vietoris27,EdelsbrunnerHarer}.
It is easy to see that $\cech{X}{r}$ is a subcomplex of $\vr{X}{2r}$.

For the case when $X = S^n$ is the $n$-sphere of diameter $\pi$ equipped with the geodesic metric, the homotopy equivalence $\cech{S^n}{r} \simeq S^n$ for $0 < r < \frac{\pi}{2}$ is known by the nerve lemma (for the statement of the nerve lemma, see~\cite[Appendix]{farber2023large} or~\cite[Corollary~4G.3]{Hatcher}).

\subsection{Graphs}
\label{sec:graphs}
All graphs in this paper are simple graphs, meaning there are no loops, and at most a single edge between any two vertices.
Formally, a graph $G=(V,E)$ consists of a set of vertices $V$ and a set $E$ of edges, where each edge is a subset of $V$ of size two.
For $v,v'\in V$, we write $v\sim v'$ to denote that the edge $\{v,v'\}$ is in $E$.

Crucially, we will consider graphs with infinitely many vertices and edges.
For example, for $n\ge 1$ and $r>0$, the 1-skeleton of the \v{C}ech complex $\cech{S^n}{r}$ is an infinite graph.

The \emph{chromatic number} of a graph $G=(V,E)$, denoted $\chi(G)$, is the minimum number of colors needed to label the vertices of $G$ such that no two adjacent vertices share the same color.

The \emph{neighborhood complex} of a graph $G=(V,E)$, denoted $N(G)$, is the simplicial complex with vertex set $V$, which contains a finite subset $\sigma\subseteq V$ as a simplex if there is some vertex $v\in V$ such that $w\sim v$ for all $w\in \sigma$.

The \emph{Borsuk graph of $S^n$} for scale $\delta>0$, denoted $\bor{S^n}{\delta}$, is a simple graph whose vertex set is $S^n$ and whose edges are all pairs of vertices $x,y\in S^n$ with $d(x,y)> \delta$.

\begin{remark}
\label{rem:opencloseborsuk}
While Borsuk graphs were originally introduced in~\cite{erdos-hajnal1967} with the open ($>$) convention that we use in this paper, Borsuk graphs with the closed ($\ge$) convention have been more popular in the literature since; see, for example,~\cite{Lovasz1978,lovasz1983self}.
The closed Borsuk graph of $S^n$ for scale $\delta>0$, denoted $\cbor{S^n}{\delta}$, is a simple graph whose vertex set is $S^n$ and whose edges are all pairs of vertices $x,y\in S^n$ with $d(x,y)\ge \delta$.
For any small $\varepsilon > 0$, we have that 
\begin{equation}
\label{eq:closeopencontainment}
\cbor{S^n}{\delta+\varepsilon} \subseteq\bor{S^n}{\delta} \subseteq \cbor{S^n}{\delta}.
\end{equation}
Therefore, many results about closed Borsuk graphs remain true for open Borsuk graphs as well.
In particular, when $\delta$ is sufficiently close to $\pi$, the equality $\chi(\bor{S^n}{\delta})=n+2$ holds.
\end{remark}

We will see in Section~\ref{sec:chromatic} that the neighborhood complex of the (open) Borsuk graph $\bor{S^n}{\delta}$ is equal to the \v{C}ech complex $\cech{S^n}{\pi-\delta}$; the neighborhood complex of the closed Borsuk graph $\cbor{S^n}{\delta}$ is a nerve complex of closed balls instead of open balls.

\section{A lower bound on connectivity}
\label{sec:lower}

The goal of this section is to provide a lower bound on $\conn(\cech{S^n}{\pi-\delta})$, the connectivity of the \v{C}ech complex of the $n$-sphere $S^n$ at the scale $\pi-\delta$, where $0<\delta<\pi$.

Using the results from Barmak's Appendix in~\cite{farber2023large}, it has been shown in~\cite[Section 4.1]{ABV} that if $0<\delta < \cov_{S^n}(2k+2)$, then the connectivity of the Vietoris--Rips complex of $S^n$ at scale $\pi-\delta$ is at least $k$, i.e., $\conn(\vr{S^n}{\pi-\delta}) \ge k$.
In the same spirit, here we use Barmak's technique to get a lower bound on $\conn(\cech{S^n}{\pi-\delta})$.

We recall that the \emph{closed star} of a vertex $v$ in a simplicial complex $K$, denoted $\st_K(v)$, is the subcomplex
$
\st_K(v) = \left\{ \sigma \in K \mid \sigma \cup \{v\} \in K \right\}
$.
A simplicial complex $K$ is said to be \emph{$\ell$-conic} if any subcomplex of at most $\ell$ vertices is contained in a simplicial cone, i.e., in the closed star $\st_K(v)$ of some vertex $v$.
Barmak's result~\cite[Appendix, Theorem~4]{farber2023large} states that if $K$ is $(2k+2)$-conic, then $K$ is $k$-connected.

\begin{proposition}
\label{prop:fintoinf}
If $0<\delta<\cov_{S^n}(2k+2)$, then $\conn(\cech{S^n}{\pi-\delta}) \ge k$.
\end{proposition}

\begin{proof}
Barmak’s Theorem~4~\cite[Appendix]{farber2023large} is proven for arbitrary simplicial complexes $K$ for which the family $\{\st_K(v)\}_{v \in K}$ is \emph{locally finite} in the sense of Bj\"{o}rner~\cite{Bjorner1995}, i.e., every vertex of $K$ belongs to only finitely many closed stars $\st_K(v)$.
This is because the version of the nerve lemma that Barmak uses is adopted from~\cite[Theorem 10.6]{Bjorner1995} where Bj\"{o}rner uses this assumption in his proof.
For $K = \cech{S^n}{\pi-\delta}$, the family $\{\st_K(v)\}_{v \in K}$ is not locally finite.
So, we adopt the following technique with finite subsets $X\subseteq S^n$ to use Barmak's result.

Let $\varepsilon>0$ be such that $\delta +\varepsilon<\cov_{S^n}(2k+2)$.
Thus, no $2k+2$ \emph{closed} balls of radius $\delta + \varepsilon$ can cover $S^n$.
Let $X$ be a finite subset of $S^n$ that is $\varepsilon$-dense, i.e., such that any open $\varepsilon$-ball in $S^n$ contains a point of $X$.
Then for any collection of $2k+2$ points $x_1,\ldots,x_{2k+2} \in X$, we have by the definition of $\cov_{S^n}(2k+2)$ that
\begin{align*}
    \bigcap_{i=1}^{2k+2}B_{S^n}(x_i;\pi-\delta-\varepsilon) 
    &= S^n \setminus \left(\bigcup_{i=1}^{2k+2}S^n\setminus B_{S^n}\left(x_i;\pi-\delta-\varepsilon\right)\right) \\
    &= S^n \setminus \left(\bigcup_{i=1}^{2k+2} B_{S^n}\left[-x_i;\delta+\varepsilon\right]\right) \\
    &\ne \emptyset.
\end{align*}
Let $y$ be a point in this non-empty intersection $\cap_{i=1}^{2k+2}B_{S^n}(x_i;\pi-\delta-\varepsilon)$.
If we increase the radius of each ball by $\varepsilon$, then we deduce that
\[
B_{S^n}(y;\varepsilon) \subseteq \bigcap_{i=1}^{2k+2}B_{S^n}(x_i;\pi-\delta).
\]
Since $X$ is $\varepsilon$-dense, there is some point $x \in X$ with 
$x \in \cap_{i=1}^{2k+2}B_{X}(x_i;\pi-\delta)$, {where we now consider balls in $X$ instead of balls in $S^n$.}
Thus, $\cech{X}{\pi-\delta}$ is $(2k+2)$-conic.
Since $X$ is finite, we use Barmak's theorem~\cite{farber2023large} to conclude that $\conn(\cech{X}{\pi-\delta}) \ge k$.

Now, for any $i \le k$, let us consider a continuous map $f\colon S^i \to \cech{S^n}{\pi-\delta}$.
To show that $\conn(\cech{S^n}{\pi-\delta})\ge k$, we show that $f$ is null-homotopic, and hence $\pi_i(\cech{S^n}{\pi-\delta})$ is trivial for $i\le k$.
We first note that the image $\text{Im}(f)$ is compact since $S^i$ is compact.
Thus, there exists a subcomplex $\cech{X}{\pi-\delta} \subseteq \cech{S^n}{\pi-\delta}$ with a finite vertex set $X$ such that $\text{Im}(f)\subseteq \cech{X}{\pi-\delta}$; {see, for instance,~\cite[Proposition~A.1]{Hatcher}.}
Without loss of generality, $X$ can be assumed to be $\varepsilon$-dense in $S^n$ because, if needed, finitely many more vertices can be added to make $X$ sufficiently dense.
Since $\conn(\cech{X}{\pi-\delta})\ge k$, we have $\pi_i(\cech{X}{\pi-\delta}) = 0$ for all $i \le k$, and thus $f$ is null-homotopic in $\cech{X}{\pi-\delta}$.
Let $H\colon S^i\times I \to \cech{X}{\pi-\delta}$ be a null-homotopy for $f$.
If $j\colon \cech{X}{\pi-\delta}\hookrightarrow\cech{S^n}{\pi-\delta}$ is the inclusion induced by the inclusion $X\hookrightarrow S^n$, then the composition
\[
    j \circ H\colon S^i \times I \to \cech{S^n}{\pi-\delta}
\]
defines a null-homotopy of $f$ in $\cech{S^n}{\pi-\delta}$.
Hence $\pi_i(\cech{S^n}{\pi-\delta}) = 0$ for all $i \le k$, and $\conn(\cech{S^n}{\pi-\delta})\ge k$.
\end{proof}

We note that $\cov_{S^n}(2k+2)$ is non-increasing in $k$, and that $\cov_{S^n}(2k+2)\to 0$ as $k\to \infty$.

\begin{remark}
Our proof technique showing that $\cech{X}{\pi-\delta}$ is $(2k+2)$-conic in fact shows something stronger, namely that any subcomplex $L$ of $\cech{X}{\pi-\delta}$ with at most $2k+2$ vertices is contained in a simplex.
Does this mean that there is potentially room for improvement to prove a stronger result than Proposition~\ref{prop:fintoinf}?

See also Section~\ref{ssec:checking-lower}, which shows that Proposition~\ref{prop:fintoinf} is not tight in the case of the circle ($n=1$).
\end{remark}

We can reinterpret the result of Proposition~\ref{prop:fintoinf} to get an explicit lower bound on $\conn(\cech{S^n}{\pi-\delta})$.

\begin{corollary}
\label{cor:lowerbound2}
For each $0<\delta <\pi$, we have that 
\[
\tfrac{1}{2}\numcov_{S^n}(\delta) - 2 \le \conn(\cech{S^n}{\pi-\delta}).
\]
\end{corollary}

\begin{proof}
Let $\conn(\cech{S^n}{\pi-\delta})=k-1$.
By Proposition~\ref{prop:fintoinf}, we have $\cov_{S^n}(2k+2)\le \delta$.
Since the infimum in the definition of the covering radius is realized (see~\cite[Lemma~3.1]{ABV}), this gives $\numcov_{S^n}(\delta)\le 2k+2$, which means $\frac{1}{2}\numcov_{S^n}(\delta)\le k+1=\conn(\cech{S^n}{\pi-\delta})+2$.
\end{proof}

\section{An upper bound on connectivity}
\label{sec:chromatic}

We now give an upper bound on the connectivity of \v{C}ech complexes of spheres.
We begin by noting that the approach used to upper bound the connectivity of Vietoris--Rips complexes of spheres will not work in the \v{C}ech case.
Indeed, the proof of~\cite[Claim~4.2]{ABV} constructed an odd map from $\vr{S^n}{\pi-\delta}$ into the Euclidean space $\R$ that missed the origin.
This is possible only since $\vr{S^n}{\pi-\delta}$ contains no edges between antipodal vertices of $S^n$ for $\delta>0$.
Unfortunately, the \v{C}ech complex $\cech{S^n}{\pi-\delta}$ does contain all antipodal edges for $\delta<\frac{\pi}{2}$, which is the regime of interest.

Therefore, we upper bound the connectivity of \v{C}ech complexes of spheres via a different approach, namely by relating this connectivity to the chromatic numbers of Borsuk graphs.
Let $\bor{S^n}{\pi-\delta}$ be the Borsuk graph whose vertex set is $S^n$ and whose edges are all pairs of vertices at a distance \emph{more than} $\pi-\delta$.
The basis for our connection is the observation that the neighborhood complex of the Borsuk graph $\bor{S^n}{\delta}$ is the same as the \v{C}ech complex $\cech{S^n}{\pi-\delta}$.
Indeed, for $\sigma\subseteq S^n$ finite, we have that
\begin{align*}
\sigma\in N(\bor{S^n}{\delta})
\Longleftrightarrow\ & \exists \hspace{1mm} v\in S^n \text{ with }x\sim v\text{ in }\bor{S^n}{\delta} \text{ for all } x\in\sigma \\
\Longleftrightarrow\ & \exists \hspace{1mm}v\in S^n \text{ with }d(x,v)> \delta \text{ for all } x\in\sigma \\
\Longleftrightarrow\ & \exists -v\in S^n \text{ with }d(x,-v)<\pi-\delta \text{ for all } x\in\sigma \\
\Longleftrightarrow\ & \sigma \in \cech{S^n}{\pi-\delta}.
\end{align*}
So, $N(\bor{S^n}{\delta})=\cech{S^n}{\pi-\delta}$.

Lov{\'a}sz's famous result~\cite[Theorem~2]{Lovasz1978} says for any graph $G$, the chromatic number satisfies $\chi(G)\ge \conn(N(G))+3$, where $N(G)$ is the neighborhood complex of $G$.

\begin{ex}
Let us illustrate Lov{\'a}sz's result for the Petersen graph.
The Petersen graph $\P$ is the Kneser graph $KG_{5,2}$, whose vertices are $2$-element subsets of a $5$-element set and two of them are joined by an edge if and only if they are disjoint.
The degree of each vertex of $\P$ is $3$, and therefore the dimension of the neighborhood complex $N(\P)$ is $2$.
Each vertex of $\P$ gives rise to a unique triangle in $N(\P)$, and no two triangles in $N(\P)$ share an edge.
Also, each vertex of $N(\P)$ is contained in exactly $3$ such triangles in $N(\P)$.
It turns out that $N(\P)$ is homotopy equivalent to a finite connected graph with holes, meaning that $N(\P)$ is not simply connected and $\conn(N(\P))=0$.
Lov{\'a}sz's bound~\cite[Theorem~2]{Lovasz1978} gives $\chi(\P)\ge \conn(N(\P))+3=0+3$, and in fact, $\chi(\P)=3$ by~\cite[Theorem~1']{Lovasz1978}.
\end{ex}

Using Lov{\'a}sz's result and the fact that $N(\bor{S^n}{\delta})=\cech{S^n}{\pi-\delta}$ for each $n\ge 1$ and $\delta \in (0,\pi)$, we obtain the following bound.

\begin{lemma}
\label{lem:lovasz}
For $0<\delta<\pi$, we have
\[\chi(\bor{S^n}{\delta})
\ge \conn(N(\bor{S^n}{\delta}))+3
= \conn(\cech{S^n}{\pi-\delta})+3.\]
\end{lemma}

\begin{figure}[bht]
\begin{center}
\includegraphics[width=2in]{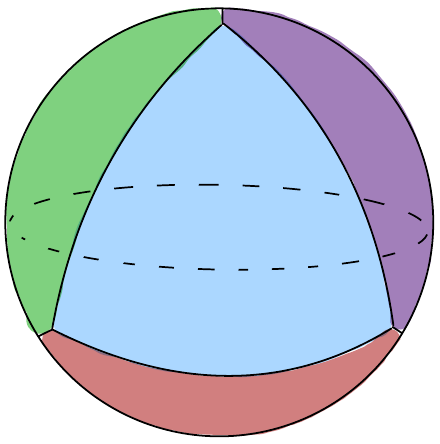}
\end{center}
\caption{We have $\chi(\bor{S^2}{\delta})=4$ for all $s_2<\delta<\pi$, as illustrated by the covering of $S^2$ by four sets each of diameter $s_2$.}
\label{fig:Sphere4coloring}
\end{figure}

The problem of bounding $\chi(\bor{S^n}{\delta})$ for $0<\delta<\pi$ close to $\pi$ is classical.
One version of the Borsuk--Ulam theorem (the Lyusternik--Schnirel'man--Borsuk covering theorem) gives, and is in fact equivalent to, the statement that $\chi(\cbor{S^n}{\delta})\ge n+2$ for all $\delta<\pi$; see, for example,~\cite{borsuk1933drei,lovasz1983self,barany1978short} and~\cite[Section~2.1, Exercise~10]{matousek2003using}.
Taking $\varepsilon > 0$ such that $\delta + \varepsilon < \pi$, the left inclusion in~\eqref{eq:closeopencontainment} gives that $\chi(\bor{S^n}{\delta})\ge n+2$, as well.
Furthermore, Raigorodskii shows that $\chi(\cbor{S^n}{\delta})=n+2$ for all $s_n<\delta<\pi$, where $s_n$ is the diameter of the radial projection of a single $n$-dimensional face of a regular $(n+1)$-dimensional simplex inscribed in $S^n$~\cite{raigorodskii2010chromatic, raigorodskii2012chromatic}.
Once again, \eqref{eq:closeopencontainment} implies that $\chi(\bor{S^n}{\delta})=n+2$ for all $s_n<\delta<\pi$; see Figure~\ref{fig:Sphere4coloring} for the case $n=2$.
From~\cite{santalo1946convex} (after plugging $\cos \ell=-\frac{1}{n+1}$ into (2.17) and (2.18) of~\cite{santalo1946convex}), we have that
\[s_n=\begin{cases} 
\arccos{\left(-\frac{n+1}{n+3}\right)} & \text{if }n \text{ is odd} \\
\arccos{\left(-\sqrt{\frac{n}{n+4}}\right)} &  \text{if }n \text{ is even}.
\end{cases}\]

The asymptotics of the chromatic numbers of Borsuk graphs of $S^n$ (though not the exact values) may be reasonably well controlled by covering and packing numbers on spheres.
See, for example,~\cite[Theorem 5.1.5]{MoyPhD}\footnote{
Moy defines his covering radius $c_{n,m}$ using balls of radius $\frac{r}{2}$, whereas we define our covering radius $\cov_{S^n}(m)$ using balls of radius $r$, giving the relationship $c_{n,m}=2\cov_{S^n}(m)$.},
which states that when $\delta > 2\cov_{S^n}(m)$ for some $m \ge 2$, then $\chi(\cbor{S^n}{\delta}) \le m$.
Since $\bor{S^n}{\delta}\subseteq \cbor{S^n}{\delta}$, we have that $\chi(\bor{S^n}{\delta}) \le m$.

\begin{proposition}
\label{prop:upper}
If $\delta > 2\cov_{S^n}(k+1)$ for some $k \ge 1$, then $\conn(\cech{S^n}{\pi-\delta}) \le k-2$.
\end{proposition}

{We recall that $\delta\in (0,\pi)$ in this paper. This excludes the possibility $k=1$ in Proposition~\ref{prop:upper}, which means that the upper bound on $\conn(\cech{S^n}{\pi-\delta})$ given above is always non-negative. 
Indeed, if $k=1$, then $2\cov_{S^n}(k+1)=\pi$ and we get the contradicting inequality $\pi-\delta<0$.}
%

\begin{proof}
Lov{\'a}sz's upper bound in Lemma~\ref{lem:lovasz} gives
$\conn(\cech{S^n}{\pi-\delta}) \le \chi(\bor{S^n}{\delta}) - 3$.
When $\delta > 2\cov_{S^n}(k+1)$ for some $k+1 \ge 2$, then $\chi(\bor{S^n}{\delta}) \le k+1$ by~\cite[Theorem 5.1.5]{MoyPhD}.
Together, these give
\[\conn(\cech{S^n}{\pi-\delta}) \le \chi(\bor{S^n}{\delta}) - 3 \le k-2.\]
\end{proof}

We note that $2\cov_{S^n}(k+1)$ is non-increasing in $k$, and that $2\cov_{S^n}(k+1)\to 0$ as $k\to \infty$.
So, for $\delta>0$, there always exists some integer $k$ such that $\delta>2\cov_{S^n}(k+1)$.

Propositions~\ref{prop:fintoinf} and~\ref{prop:upper} together prove Theorem~\ref{thm:main}.

We can reinterpret the result of Proposition~\ref{prop:upper} to get an explicit upper bound on $\conn(\cech{S^n}{\pi-\delta})$.

\begin{corollary}\label{cor:upperbound2}
For each $0<\delta <\pi$, we have that 
\[
\conn(\cech{S^n}{\pi-\delta})\le \numcov_{S^n}\left(\tfrac{\delta}{2}\right)-2.
\]
\end{corollary}

\begin{proof}
Let $\conn(\cech{S^n}{\pi-\delta})=k-1$.
The contrapositive of Proposition~\ref{prop:upper} gives $\delta\le 2\cov_{S^n}(k+1)$, or $\frac{\delta}{2}\le \cov_{S^n}(k+1)$.
This gives $k+1\le \numcov_{S^n}(\tfrac{\delta}{2})$, which means $k-1\le \numcov_{S^n}(\tfrac{\delta}{2})-2$.
\end{proof}

Corollaries~\ref{cor:lowerbound2} and~\ref{cor:upperbound2} together prove Corollary~\ref{cor:main}.

\section{The case of the circle}
\label{sec:circle}

In the case of the circle, we know the homotopy types of $\cech{S^1}{r}$ for all $r$; see~\cite[Section~9]{AA-VRS1}.
This gives us one ``data point'' ($n=1$) that any more general results for $\cech{S^n}{r}$ with $n\ge 1$ will need to generalize.
See Figure~\ref{fig:barsS1}.

By~\cite[Theorem 9.8]{AA-VRS1}, the homotopy types of $\cech{S^1}{r}$ are all odd spheres.
It turns out that
\begin{equation}
\label{eq:S1cech}
\cech{S^1}{r} \simeq S^{2k+1} \hspace{5mm} \text{if} \hspace{5mm} r \in \left(\frac{\pi k}{k+1},\frac{\pi (k+1)}{k+2}\right],
\end{equation}
with $\cech{S^1}{r}$ contractible for all $r\ge \pi$.
From this, we get that
\begin{equation}
\label{eq:S1conn}
\conn(\cech{S^1}{\pi-\delta}) = 2k \hspace{5mm} \text{if} \hspace{5mm} \delta \in \left[\frac{\pi}{k+2},\frac{\pi}{k+1}\right).
\end{equation}

\subsection{Comparing the lower bound}
\label{ssec:checking-lower}

We check the lower bound to $\conn(\cech{S^n}{\pi-\delta})$ obtained in Section~\ref{sec:lower} in the case $n=1$.

From Proposition~\ref{prop:fintoinf}, we have that if $\conn(\cech{S^1}{\pi-\delta}) \le k-1$, then $\delta\ge\cov_{S^1}(2k+2)$.
Via the substitution $k\mapsto 2k+1$, this means that if $\conn(\cech{S^1}{\pi-\delta}) \le 2k$, then $\delta \ge \cov_{S^1}(4k+4) = \frac{\pi}{4k+4}$.
Using the exact connectivity values of $\cech{S^1}{\pi-\delta}$ in~\eqref{eq:S1conn}, we compare the lower bounds to $\delta$, namely $\frac{\pi}{4k+4}$ and $\frac{\pi}{k+2}$, and see that the bound on $\delta$ given by Proposition~\ref{prop:fintoinf} is off by a multiplicative factor of nearly $4$ for even connectivity when $n=1$.
Hence, the lower bound to $\conn(\cech{S^n}{\pi-\delta})$ from Corollary~\ref{cor:lowerbound2} is not sharp for $n=1$.

\begin{remark}
    From~\cite[Theorem 7.4]{AA-VRS1}, we have that 
    \[
\vr{S^1}{r} \simeq S^{2k+1} \hspace{5mm} \text{if} \hspace{5mm} r \in \left(\frac{2\pi k}{2k+1},\frac{2\pi (k+1)}{2k+3}\right].
\]
From this, we get that
\[
\conn(\vr{S^1}{\pi-\delta}) = 2k \hspace{5mm} \text{if} \hspace{5mm} \delta \in \left[\frac{\pi}{2k+3},\frac{\pi}{2k+1}\right).
\]
From~\cite[Claim 4.1]{ABV} in the case $n=1$, if $\conn(\vr{S^1}{\pi-\delta}) \le 2k$, then $\delta \ge \cov_{S^1}(4k+4) = \frac{\pi}{4k+4}$.
So, the lower bound on $\delta$ for Vietoris--Rips complexes of the circle is off by nearly a factor of $2$  for even connectivity.
It is interesting that Barmak's result, which is not specific to clique complexes, gives better lower bounds for $\conn(\vr{S^1}{\pi-\delta})$ than for $\conn(\cech{S^1}{\pi-\delta})$.
We do not yet have a conceptual understanding of why this should be the case.
\end{remark}

\subsection{Comparing the upper bound}
\label{sec:compareupperbound}
We now check the upper bound to $\conn(\cech{S^n}{\pi-\delta})$ obtained in Section~\ref{sec:chromatic} in the case $n=1$.
To do that, we first note the following.

\begin{lemma}
\label{lem:chromatic_num}
$\chi(\bor{S^1}{\delta}) = \bigl\lceil\frac{2\pi}{\delta}\bigr\rceil$ for any $0<\delta< \pi$.
\end{lemma}

\begin{proof}
We know that $\chi(\cbor{S^1}{\delta}) = \bigl\lceil\frac{2\pi}{\delta}\bigr\rceil$; see, for example,~\cite[Section 5.1.2]{MoyPhD}.
So, for $k\geq 2$, $\chi(\cbor{S^1}{\delta}) = k$ when $\tfrac{2\pi}{k} \leq \delta < \tfrac{2\pi}{k-1}$.
Let $\tfrac{2\pi}{k} \leq \delta < \tfrac{2\pi}{k-1}$, and choose $\varepsilon > 0$ such that $\delta + \varepsilon < \tfrac{2\pi}{k-1}$.
Then by~\eqref{eq:closeopencontainment}, $k=\chi(\cbor{S^1}{\delta+\varepsilon}) \leq \chi(\bor{S^1}{\delta}) \leq \chi(\cbor{S^1}{\delta})= k$.
Thus, we have that $\chi(\bor{S^1}{\delta}) = \chi(\cbor{S^1}{\delta})= \bigl\lceil\frac{2\pi}{\delta}\bigr\rceil$.
\end{proof}

Let us compare both Lov{\'a}sz's bound (Lemma~\ref{lem:lovasz}) and Proposition~\ref{prop:upper}, which give upper bounds on $\conn(\cech{S^n}{\pi-\delta})$, to the exact connectivity values when $n=1$.

\begin{itemize}
\item
Plugging $n=1$ into Lov{\'a}sz's bound $\conn(\cech{S^n}{\pi-\delta}) \le \chi(\bor{S^n}{\delta}) - 3$, and using the exact chromatic numbers of the Borsuk graphs $\bor{S^1}{\delta}$ from Lemma~\ref{lem:chromatic_num}, together give $\conn(\cech{S^1}{\pi-\delta}) \le \bigl\lceil\tfrac{2\pi}{\delta}\bigr\rceil - 3$.
If $\conn(\cech{S^1}{\pi-\delta}) = 2k$, then $2k+3 \le \bigl\lceil\tfrac{2\pi}{\delta}\bigr\rceil$, which implies that $\delta < \tfrac{2\pi}{2k+2} = \tfrac{\pi}{k+1}$.
\item
Proposition~\ref{prop:upper} states if $\conn(\cech{S^n}{\pi-\delta}) \ge k-1$, then $\delta\le 2\cov_{S^n}(k+1)$.
For $n=1$, if $\conn(\cech{S^1}{\pi-\delta}) \ge k-1$, then $\delta\le 2\cov_{S^1}(k+1) = \tfrac{2\pi}{k+1}$.
Via the substitution $k\mapsto 2k+1$, we get that if $\conn(\cech{S^1}{\pi-\delta}) \ge 2k$, then $\delta\le \tfrac{2\pi}{2k+2} = \tfrac{\pi}{k+1}$.
\item
By~\eqref{eq:S1conn}, $\conn(\cech{S^1}{\pi-\delta}) = 2k$ if and only if $\delta \in \bigl[\tfrac{\pi}{k+2},\tfrac{\pi}{k+1}\bigr)$.
\end{itemize}

Lov{\'a}sz's bound $\conn(\cech{S^n}{\pi-\delta}) \le \chi(\bor{S^n}{\delta}) - 3$ and Proposition~\ref{prop:upper} both give the optimal upper bound of $\tfrac{\pi}{k+1}$ on $\delta$, for even connectivity, when $n=1$.
The former excludes the possibility $\delta=\tfrac{\pi}{k+1}$, whereas the latter does not, and hence, the former upper bound is slightly better.
This makes sense since our Proposition~\ref{prop:upper} is derived from Lov{\'a}sz's bound but does not use exact chromatic numbers of Borsuk graphs.

\section{Homological dimension}
\label{sec:hdim}

In this section, we explore the largest dimension in which the homology of $\cech{S^n}{r}$ is non-zero.
For a topological space $Y$ and a commutative ring $R$ with unity, the $R$-homological dimension of $Y$ is defined as 
\[
\hdim_R(Y) \coloneqq \max\{i\in\Z \mid \widetilde{H}_i(Y;R) \ne 0\}.
\]

The $\Z$-homological dimension of a space is informative of its topology: a space that is homotopy equivalent to a CW complex of dimension $k$ has $\Z$-homological dimension at most $k$.
The $\Z$-homological dimension of a space is often connected to its homotopy connectivity.
Indeed, if $Y$ is simply connected, then by the Hurewicz theorem~\cite[Theorem~4.32]{Hatcher}, we have $\conn(Y) \le \hdim_{\Z}(Y)-1$.
Note that $\conn(Y)$ need not\footnote{
To see that we may have $\hdim_{\Z}(X) > \conn(X) > \hdim_{\Z_2}(X)$, let $p$ be an odd prime, and let $X = M(\Z_p,r+1)$ be a Moore space that is $r$-connected for some $r>0$.
Then $\hdim_{\Z}(X)=r+1>r>0=\hdim_{\Z_2}(X)$.
}
be a lower bound to $\hdim_R(Y)-1$ for a general ring $R$.

We know that $\cech{S^n}{\pi-\delta}\simeq S^n$ for all $\delta\in[\tfrac{\pi}{2},\pi)$ due to the nerve lemma.
We also know that $\cech{S^1}{\pi-\delta}$ is homotopic to some odd-dimensional sphere for each $\delta>0$.
Hence, the $\Z$-homological dimension of $\cech{S^n}{\pi-\delta}$ is known in all these cases.
So, let us restrict our attention to the case when $n\ge 2$ and $\delta\in(0,\tfrac{\pi}{2})$.

In Proposition \ref{prop:fintoinf}, we gave a lower bound on the connectivity $\conn(\cech{S^n}{\pi-\delta})$.
In particular, if $0<\delta<\cov_{S^n}(2k+2)$, then $\conn(\cech{S^n}{\pi-\delta}) \ge k$.
For $n\ge 2$, it follows from~\cite[Theorem 11.2]{virk20201} that the \v{C}ech complex $\cech{S^n}{\pi-\delta}$ is simply connected for each $\delta\in (0,\tfrac{\pi}{2})$.
Hence, using the relation between connectivity and $\Z$-homological dimension stated in the previous paragraph, we get a lower bound to $\hdim_{\Z}(\cech{S^n}{\pi-\delta})$:
if $0<\delta<\cov_{S^n}(2k+2)$, then $\hdim_{\Z}(\cech{S^n}{\pi-\delta})\geq k+1$.
In other words, for any $\delta \in (0,\tfrac{\pi}{2})$ and $n\ge 2$, Corollary~\ref{cor:lowerbound2} gives that
\begin{equation}
\label{eq:h-dimBarmak}
\hdim_{\Z}(\cech{S^n}{\pi-\delta})\ge \tfrac{1}{2}\numcov_{S^n}(\delta) - 1.
\end{equation}

Next, we try to use ideas from~\cite[Lemma~4.3]{MatsushitaWakatsuki} to get a lower bound to $\hdim_{\Z}(\cech{S^n}{\pi-\delta})$ better than that provided in~\eqref{eq:h-dimBarmak}.
We consider the $\Z_2$-homological dimension.
Using the universal coefficient theorem for homology~\cite[Theorem~3A.3]{Hatcher}, one can see that for any space $Y$, we have $\hdim_{\Z}(Y)\ge \hdim_{\Z_2}(Y)-1$.
In fact, if $H_*(Y;\Z)$ is torsion-free, then $\hdim_{\Z}(Y)\ge \hdim_{\Z_2}(Y)$.
So, a lower bound to $\hdim_{\Z_2}(\cech{S^n}{\pi-\delta})$ will give us a lower bound to $\hdim_{\Z}(\cech{S^n}{\pi-\delta})$.

To potentially lower bound $\hdim_{\Z_2}(\cech{S^n}{\pi-\delta})$, the theory of box complexes will be relevant.
Box complexes of \emph{finite} graphs have been well-studied; we refer the reader to~\cite{MZ,csorba2007homotopy,MatsushitaWakatsuki}.

There are several related definitions of box complexes {(see, for instance,~\cite{MZ, csorba2007homotopy})}, which all share similar properties.
{In particular, we are using the following definition of box complex.
For a \emph{finite} graph $G$, the box complex $\cB(G)$ is defined to be
\begin{multline*}
\cB(G) = \{A \uplus B \,\mid\ A,B \subset V(G),\ A \cap B = \emptyset,\ G[A,B]\text{ is complete bipartite},\ \CN(A)\neq \emptyset \neq \CN(B)\}.
\end{multline*}
Here, $A\uplus B=\{(a,0)\mid a\in A\}\cup \{(b,1)\mid b\in B\}$, $\CN(A)=\{v\in V(G)\mid (v,a)\in E(G)\ \text{for all }a\in A\}$, {$\CN(B)=\{v\in V(G)\mid (v,b)\in E(G)\ \text{for all }b\in B\}$,} and $G[A,B]$ is a (not necessarily induced) subgraph of $G$ with vertex set $A\cup B$ and edge set $\{(a,b)\in E(G)\mid a\in A, b\in B\}$, where $V(G)$ and $E(G)$ denote the sets of vertices and edges of $G$, respectively.}

For a finite graph $G$, the box complex $\cB(G)$ is a free $\Z_2$-simplicial complex
satisfying the following three properties:
\begin{enumerate}
\item[(a)] The box complex $\cB(G)$ and the neighborhood complex $N(G)$ are homotopy equivalent.
\item[(b)] The box complex $\cB(K_m)$ of the complete graph $K_m$ is $\Z_2$-homotopy equivalent to the sphere $S^{m-2}$.
\item[(c)] If $H$ is a subcomplex of $G$, then $\cB(H)$ is a $\Z_2$-subcomplex of $\cB(G)$.
\end{enumerate}

Since box complexes have been studied on finite graphs, we restrict ourselves to finite subsets $X$ of $S^n$ and work with them, in the hope that the behavior of \v{C}ech complexes of finite subsets governs the behavior of \v{C}ech complexes of the entire sphere as the finite subset grows denser and denser.
Instead of the (intrinsic) \v{C}ech complex $\cech{X}{r}\coloneqq \mathrm{Nerve}(\{B_{X}(x;r)\}_{x\in X})$ defined using balls in $X$, we emphasize that moving forward, we will instead work with the \emph{ambient} \v{C}ech complex $\acech{X}{S^n}{r}\coloneqq \mathrm{Nerve}(\{B_{S^n}(x;r)\}_{x\in X})$ defined using balls in $S^n$.
That is, $\sigma \subseteq X$ is a simplex in $\acech{X}{S^n}{r}$ if the intersection $\cap_{x\in \sigma}B_{S^n}(x,r)$ of balls in $S^n$ is nonempty.

We define the (open) Borsuk graph of $X\subseteq S^n$ for scale $\delta>0$ similarly, i.e., the vertices of $\bor{X}{\delta}$ are the points in $X$ and there is an edge between $v, w\in X$ if $d(v, w)> \delta$.
Here, the metric $d$ is the geodesic distance, derived from $S^n$.
We use $\acech{X}{S^n}{r}$ instead of $\cech{X}{r}$ since $N(\bor{X}{\delta}) = \acech{X}{S^n}{\pi-\delta}$.

To bound $\hdim_{\Z_2}(\acech{X}{S^n}{r})$ from below for finite $X\subseteq S^n$, we will also use the fact that if $Y$ is a finite free $\Z_2$-simplicial complex and if there is a $\Z_2$-map from $S^k$ to $Y$, then $\hdim_{\Z_2}(Y)\ge k$; see~\cite[Lemma~7]{matsushita2023dominance} or~\cite[Lemma~4.4]{MatsushitaWakatsuki}.
We also use the packing number from Section~\ref{ssec:metric}.

\begin{proposition}\label{z2hdim}
Let $X$ be a finite subset of $S^n$ and let $\delta > 0$.
Then
\[
\hdim_{\Z_2}(\acech{X}{S^n}{\pi-\delta}) \geq \numpack_{X}(\delta)-2.
\]
\end{proposition}

\begin{proof}
Let $\omega(G)$ denote the clique number of a finite graph $G$, which is the number of vertices in the largest clique of $G$.
So, $K_{\omega(G)}$ is a subcomplex of the graph $G$, and thus, $\cB(K_{\omega(G)})\xhookrightarrow{\Z_2} \cB(G)$.
Also, $\cB(K_{\omega(G)})\simeq_{\Z_2}S^{\omega(G)-2}$.
Hence, there exists a $\Z_2$-map from $S^{\omega(G)-2}$ to $\cB(K_{\omega(G)})$.
Since $\cB(K_{\omega(G)})$ is a $\Z_2$-subcomplex of $\cB(G)$, there exists a $\Z_2$-map from $S^{\omega(G)-2}$ to $\cB(G)$.
Therefore, from~\cite[Lemma~4.4]{MatsushitaWakatsuki}, we conclude that $\hdim_{\Z_2}(\cB(G))\ge \omega(G)-2$.
Taking $G=\bor{X}{\delta}$ for a finite subset $X\subset S^n$, this gives 
\begin{align*}
\hdim_{\Z_2}(\acech{X}{S^n}{\pi-\delta})
&= \hdim_{\Z_2}(N(\bor{X}{\delta})) \\
&= \hdim_{\Z_2}(\cB(\bor{X}{\delta})) \\
&\ge \omega(\bor{X}{\delta})-2 \\
&= \numpack_{X}(\delta)-2.
\end{align*}
The equality $\omega(\bor{X}{\delta}) = \numpack_X(\delta)$ used in the last line above follows from the definitions.
\end{proof}

Using Proposition~\ref{z2hdim} and the relationship between $\hdim_{\Z}(Y)$ and $\hdim_{\Z_2}(Y)$ stated above, we get the following.

\begin{corollary}\label{zhdim}
For a finite subset $X\subseteq S^n$ and $\delta > 0$, we have that
\[
\hdim_{\Z}(\acech{X}{S^n}{\pi-\delta}) \geq \numpack_{X}(\delta)-3.
\]
Furthermore, if $H_*(\acech{X}{S^n}{\pi-\delta};\Z)$ is torsion-free for some $\delta>0$ and $X\subset S^n$, then we have that
\[
\hdim_{\Z}(\acech{X}{S^n}{\pi-\delta}) \geq \numpack_{X}(\delta)-2.
\]
\end{corollary}

What about the case when $X=S^n$ is infinite?
Given an infinite graph $G$, while we can define its box complex $\cB(G)$, we do not know if $\cB(G)$ satisfies the same properties as those mentioned above for the box complexes of finite graphs.

\begin{question}
\label{ques:lower-bound-hdim}
If the results for finite graphs were true for certain infinite graphs (such as $\bor{S^n}{\delta}$) as well, we might get that 
\begin{equation}
\label{eq:h-dimMatsushita}
\hdim_{\Z}(\cech{S^n}{\pi-\delta}) \ge \numpack_{S^n}(\delta)-3.
\end{equation}
Is this actually the case? 
\end{question}

We note that if~\eqref{eq:h-dimMatsushita} is true and, additionally, if the integral homology of $\cech{S^n}{\pi-\delta}$ is torsion-free for some $n\ge 1$ and $\delta>0$, then we have that
\begin{equation}
\label{eq:h-dimMatsushita2}
\hdim_{\Z}(\cech{S^n}{\pi-\delta}) \ge \numpack_{S^n}(\delta)-2.
\end{equation}

It seems that the proof technique of~\cite[Lemma~7]{matsushita2023dominance} might offer tools towards proving~\eqref{eq:h-dimMatsushita} (and consequently~\eqref{eq:h-dimMatsushita2}).
Indeed, proceeding as in the proof of~\cite[Lemma~7]{matsushita2023dominance}, we get that for an infinite $\Z_2$-complex $Y$ equipped with a $\Z_2$-map $S^k \to Y$, either $\hdim_{\Z_2}(Y) \ge k$ or $\hdim_{\Z_2}(\tfrac{Y}{\Z_2})$ is infinite.

\begin{question}
In our context, can we show that $\hdim\left(\frac{\cB(\bor{S^n}{\delta})}{\Z_2}\right)<\infty$?
\end{question}

\begin{remark}
Let us see how good of a bound~\eqref{eq:h-dimMatsushita2} would be, in the case of the \v{C}ech complexes of a circle.
By ~\cite[Theorem 9.8]{AA-VRS1} and~\eqref{eq:S1cech}, we have
\[
\hdim_{\Z}(\cech{S^1}{\pi-\delta}) = 2k+1 \hspace{5mm} \text{if} \hspace{5mm} \delta \in \left[\frac{\pi}{k+2},\frac{\pi}{k+1}\right).
\]
Also, note that $\numpack_{S^1}(\delta) = k$ when $\delta \in [\tfrac{2\pi}{k+1}, \tfrac{2\pi}{k} )$.
Thus, if $\hdim_{\Z}(\cech{S^1}{\pi-\delta}) \leq 2k+1$, then according to~\eqref{eq:h-dimMatsushita2}, we would have $\numpack_{S^n}(\delta) \leq 2k +3$.
This implies that $\delta \ge \frac{2\pi}{2k +4}$ or $\delta \ge \frac{\pi}{k +2}$, which is exactly the lower bound given by ~\cite[Theorem 9.8]{AA-VRS1}.
Hence, the lower bound in~\eqref{eq:h-dimMatsushita2} would be tight in the case of the circle, and also a $4$-fold improvement over the prior lower bound~\eqref{eq:h-dimBarmak}.
However, we do not know whether results similar to that of~\cite[Lemma~4.3]{MatsushitaWakatsuki} hold for infinite graphs.
\end{remark}

\section{Conclusion}
\label{sec:conclusion}

We conclude with a list of open questions and conjectures, that, we hope, will stimulate further interest and research.

\begin{conjecture}
\label{conj:non-decreasing}
For $n\ge 1$, the connectivity of the \v{C}ech complex $\cech{S^n}{r}$ is a non-decreasing function of the scale $r\in(0,\pi)$.
\end{conjecture}

Conjecture~\ref{conj:non-decreasing} is true for $n=1$~\cite[Section~9]{AA-VRS1}, and for $n\ge 2$, it is known that $\cech{S^n}{r}$ is simply connected for all $r$.
For $n\ge 3$, even though though $\cech{S^n}{r}\simeq S^n$ has connectivity $n-1$ for $0<r<\frac{\pi}{2}$, it is not known if $\conn(\cech{S^n}{r})\ge n-1$ for $\frac{\pi}{2}\le r<\pi$.

\begin{conjecture}
\label{conj:countably}
For $n\ge 1$, the homotopy type of the \v{C}ech complex $\cech{S^n}{r}$ changes only countably many times as $r$ varies over the range $(0,\pi)$.
\end{conjecture}

\begin{question}
Can we develop a Morse theory to describe the \v{C}ech complexes of Riemannian manifolds?
\end{question}

If so, Conjecture~\ref{conj:countably}, which is true in the case of the circle $(n=1)$, could be proven by showing that in such a Morse theory, there are only countably many critical events for $\cech{S^n}{r}$ as $r$ increases.

\begin{conjecture}\label{conj:finite}
For each dimension $n\ge 1$ and scale $r>0$, the \v{C}ech complex $\cech{S^n}{r}$ is homotopy equivalent to a finite CW complex.
\end{conjecture}

We know that this conjecture is true for $n=1$ and all $r>0$, and for all $n\ge 1$ when $r<\frac{\pi}{2}$.

We note that Conjecture~\ref{conj:finite} is false if one defines the \v{C}ech complexes using closed balls instead of open balls; see~\cite[Section~9]{AA-VRS1}.

\begin{question}
Assuming that $\cech{S^n}{r}$ is homotopy equivalent to a finite-dimensional CW complex, how does the dimension vary as a function of the scale $r$?
\end{question}

If $\cech{S^n}{r}$ were homotopy equivalent to a $k$-dimensional CW complex, then $\hdim_{\Z}(\cech{S^n}{r})\le k$.
Dually, if $\hdim_{\Z}(\cech{S^n}{r}) > k$ (see Question~\ref{ques:lower-bound-hdim}), then $\cech{S^n}{r}$ could only be homotopy equivalent to a CW complex of dimension larger than $k$.

\begin{question}
What assumptions are needed on a compact Riemannian manifold $M$ so that $\cech{M}{r}$ is homotopy equivalent to a finite (or to a finite-dimensional) CW complex?
\end{question}

\begin{question}
Can the lower bound in Proposition~\ref{prop:fintoinf} be improved?
\end{question}

\begin{question}
What upper and lower bounds can one give on the integral homological dimension of \v{C}ech complexes and Vietoris--Rips complexes of spheres?
\end{question}

\section*{Acknowledgements}

HA was funded in part by the Simons Foundation’s Travel Support for Mathematicians program.
The authors would like to thank Anurag Singh and Samir Shukla for conversations related to Section~\ref{sec:hdim}.
The authors have no competing interests to declare that are relevant to the content of this article.

\bibliographystyle{plain}
\bibliography{CechConnectivity}

\bigskip

\end{document}